\documentclass[reqno,letterpaper,11pt]{amsart}

\usepackage{graphicx,color,epsfig}
\usepackage{amsfonts,amsmath,amssymb,amsthm}
\usepackage{hyperref}
\usepackage{comment}

\newtheorem{thm}{Theorem}[section]

\newtheorem{lem}[thm]{Lemma}
\newtheorem{cor}[thm]{Corollary}

\theoremstyle{remark}
\newtheorem{rem}[thm]{Remark}

\newtheorem{ex}{Example}[section]

\theoremstyle{definition}
\newtheorem{defn}{Definition}

\newcommand{\ra}{\rightarrow}

\newcommand{\Z}{\mathbb Z}     

\renewcommand{\d}{\delta}

\newcommand{\e}{\varepsilon}

\newcommand{\s}{\sigma}

\newcommand{\ind}[1]{ \mathbf{1}_{ \{ #1 \} } } 

\newcommand{\be}{\begin{equation}}
\newcommand{\ee}{\end{equation}}

\newcommand{\Pv}{\mathbf{P}}		

\newcommand{\p}{\vec{p}}
\newcommand{\q}{\vec{q}}

\newcommand{\cleq}{\preccurlyeq}		

\begin{document}

\title[Monotonicity in excited random walks]{Strict monotonicity properties in one-dimensional excited random walks}
\author{Jonathon Peterson}
\address{Jonathon Peterson \\  Purdue University \\ Department of Mathematics \\ 150 N University Street \\ West Lafayette, IN  47907 \\ USA}
\email{peterson@math.purdue.edu}
\urladdr{http://www.math.purdue.edu/~peterson}


\date{\today}

\begin{abstract}
 We consider excited random walks with $M$ ``cookies'' where the $i$-th cookie at each site has strength $p_i$. There are certain natural monotonicity results that are known for the excited random walk under some partial orderings of the cookie environments. For instance the limiting speed $\lim_{n\ra\infty} X_n/n = v(p_1,p_2,\ldots, p_M)$ is increasing in each $p_j$. We improve these monotonicity results to be strictly monotone under a partial ordering of cookie environments introduced by Holmes and Salisbury. While the self-interacting nature of the excited random walk makes a direct coupling proof difficult, we show that there is a very natural coupling of the associated branching process from which the monotonicity results follow.  
\end{abstract}

\maketitle

\section{Introduction}
An excited random walk is an interacting random walk where the transition probabilities depend on the number of prior visits to the current site. In this paper we will be interested in the case of an excited random walk on $\Z$ with a deterministic ``cookie environment'' with finitely many cookies at each site. The model can be described in the following way. Let $M$ be a positive integer, and let $\vec{p} = (p_1,p_2,\ldots, p_M) \in (0,1)^M$. 
Then, an excited random walk $\{X_n\}_{n\geq 0}$ is a stochastic process with law $P_{\p}$ given by $P_{\p} (X_0 = 0) = 1$ and 
\begin{align*}
 P_{\p}(X_{n+1} = X_n + 1 \, | \, \mathcal{F}_n ) 
& = 1 - P_{\p}(X_{n+1} = X_n - 1 \, | \, \mathcal{F}_n ) \\
&=
\begin{cases}
 p_j & \text{if } \#\{k\leq n: X_k = X_n \} = j \\
 1/2 & \text{if } \#\{k\leq n: X_k = X_n \} > M. 
\end{cases}
\end{align*}
with $\mathcal{F}_n = \s(X_k, \, k\leq n)$.
The following ``cookie'' interpretation of excited random walks was originally given by Zerner \cite{zMERW}. One imagines that there is a stack of $M$ cookies at each site $x \in \Z$. Then, upon the $j$-th visit to $x$ the random walker eats the $j$-th cookie at that site which ``excites'' the random walker and causes the next step to $x+1$ with probability $p_j$ and to $x-1$ with probability $1-p_j$. After $M$ visits to $x$ there will be no cookies left at the site and so upon all subsequent returns to $x$ the random walker moves like a simple symmetric random walk. The sequence $\vec{p} = (p_1,p_2,\ldots p_M)$ is therefore called the \emph{cookie environment} for the excited random walk.

\begin{rem}
 In the most general model for excited random walks the cookie environment can be random, non-spatially homogeneous, and with infinitely many cookies per site \cite{zMERW,kzPNERW}. However, for the purposes of this paper we will consider only the simpler setup of deterministic cookie environments with finitely many cookies per site. 
\end{rem}
 
Excited random walks in the above setup have been studied extensively, and there is an explicit criterion for recurrence/transience as well as ballisticity that depends only on the total drift of cookies at each site:
 \be\label{ddef}
\d = \d(\vec{p}) = \sum_{j=1}^M (2 p_j - 1). 
\ee
In particular, the following is known. 
\begin{thm}[Theorems 1 and 2 in \cite{kzPNERW}]\label{rtspeed}
Let $X_n$ be an excited random walk with cookie environment $\vec{p} = (p_1, p_2,\ldots, p_M)$ with $p_j \in (0,1)$ for all $j =1,2,\ldots M$. 
\begin{enumerate}
 \item $P_{\p}( \lim_{n\ra\infty} X_n = +\infty) = 1$ if and only if $\d> 1$ \label{transright}
 \item $P_{\p}( \lim_{n\ra\infty} X_n = -\infty) = 1$ if and only if $\d< -1$
 \item $P_{\p}( \liminf_{n\ra\infty} X_n = -\infty, \,  \limsup_{n\ra\infty} X_n = \infty) = 1$ if and only if $\d\in [-1,1]$
 \item There exists a deterministic constant $v(\vec{p})$ such that 
\[
 \lim_{n\ra\infty} \frac{X_n}{n} = v(\vec{p}), \quad P_{\p}\text{-a.s.}
\]
Moreover, $v(\vec{p}) = 0$ if and only if $\d \in [-2,2]$. 
\end{enumerate}
\end{thm}

Our aim will be to study some natural monotonicity properties of some quantities defined in terms of the excited random walk. To this end, first consider the natural partial ordering of cookie environments given by 
\be\label{natorder}
 \vec{p} \leq \vec{q} \quad\text{if} \quad p_i \leq q_i \text{ for all } i=1,2,\ldots M. 
\ee
The following monotonicity results follow from slightly more general statements in \cite{zMERW}.
\begin{thm}[Theorem 16 in \cite{zMERW}]\label{zTransmono}
 If $\vec{p} \leq \vec{q}$, then $P_{\p}( X_n > 0, \, \forall n>0) \leq P_{\q}( X_n > 0, \, \forall n>0)$. 
\end{thm}

\begin{thm}[Theorem 17 in \cite{zMERW}]\label{zSpeedmono}
 If $\vec{p} \leq \vec{q}$, then $v(\p) \leq v(\q)$.
\end{thm}

\begin{rem}
In \cite{zMERW}, Zerner assumed that all cookies induced a positive drift. That is $p_j\geq 1/2$ for all $j$. However, the proofs of Theorems \ref{zTransmono} and \ref{zSpeedmono} do not use this assumption.  
\end{rem}

Recently, Holmes and Salisbury introduced a weaker partial ordering on cookie environments than \eqref{natorder} which generalizes Zerner's monotonicity results above \cite{hsCombinatorial}. 

\begin{defn}\label{coupledef}
 We will write $\p \cleq \q$ if there exists a coupling $\Pv$ of  $(\mathbf{Y}, \mathbf{Z})$ with $\mathbf{Y} = (Y_1,Y_2,\ldots,Y_M)$ and $\mathbf{Z} = (Z_1, Z_2,\ldots,Z_M)$ are such that 
\begin{itemize}
 \item $\{Y_1,Y_2,\ldots, Y_M\}$ are independent Bernoulli random variables with $Y_j \sim \text{Ber}(p_j)$. 
 \item $\{Z_1,Z_2,\ldots, Z_M\}$ are independent Bernoulli random variables with $Z_j \sim \text{Ber}(q_j)$. 
 \item $\Pv\left( \sum_{j=1}^m Y_j \leq \sum_{j=1}^m Z_j \right) = 1$ for every $m=1,2,\ldots M$.
\end{itemize}
Moreover, we will write $\p \prec \q$ if in the above coupling we have $\Pv\left( \sum_{j=1}^m Y_j < \sum_{j=1}^m Z_j \right) > 0$ for some $m \leq M$. 
\end{defn}

\begin{rem}
 For any fixed $\p$ and $\q$, it is not immediately obvious how to check the relation $\p \cleq \q$ using the definition above. However, checking for the existence of a coupling as required in the definition is simply an algebraic (but somewhat tedious) computation. 
For instance, when $M=2$ it can be shown that 
\begin{equation}\label{wocrit}
 \p = (p_1,p_2) \cleq (q_1,q_2) = \q 
\quad \iff \quad 
\begin{cases}
 p_1 \leq q_1 \\
 p_1 p_2 \leq q_1 q_2 \\
 (1-p_1)(1-p_2) \geq (1-q_1)(1-q_2). 
\end{cases}
\end{equation}
Obviously, the conditions on the right above are necessary since any coupling as in Definition \ref{coupledef} must have $\Pv(Y_1 = 1) \leq \Pv(Z_1 = 1)$, $\Pv(Y_1+Y_2 = 2) \leq \Pv(Z_1+Z_2=2)$, and $\Pv(Y_1+Y_2=0) \geq \Pv(Z_1+Z_2=0)$.
To see that the conditions are also sufficient one needs to construct a coupling as in Definition \ref{coupledef}. 
A coupling is easy to construct if one also has $p_2 \leq q_2$ since one can independently couple $Y_1 \leq Z_1$ and $Y_2\leq Z_2$. On the other hand, if $p_2 > q_2$ then one can use the following explicit coupling. 
\[
\begin{array}{|c|c|c|}
\hline
 \mathbf{y} & \mathbf{z} & \Pv( \mathbf{Y}=\mathbf{y}, \, \mathbf{Z}=\mathbf{z} ) \\
\hline
 (0,0) & (0,0) & (1-q_1)(1-q_2) \\
 (0,0) & (1,0) & (1-p_1)(1-p_2) - (1-q_1)(1-q_2) \\
 (1,0) & (1,0) & p_1(1-p_2) \\
 (0,1) & (1,0) & p_2 - q_2 \\
 (0,1) & (0,1) & (1-q_1)q_2 \\
 (0,1) & (1,1) & q_1 q_2-p_1 p_2 \\
 (1,1) & (1,1) & p_1 p_2 \\
\hline
\end{array}
\]
\end{rem}

\begin{ex}
 It was noted in \cite{hsCombinatorial} that if one obtains $\q$ from $\p$ by re-ordering the coordinates of $\p$ to be non-decreasing, then $\p \cleq \q$. This can be verified by a series of swaps of coordinates since $(p_1,p_2) \cleq (p_2,p_1)$ if $p_1\leq p_2$. Similarly, using the criterion \eqref{wocrit} one can show that 
\[
(0.9,0.85,0.8) \prec (0.9,0.88,0.78) 
\prec (0.92,0.88,0.77).
\]
\end{ex}

\begin{rem}\label{ineqrem}
 Note that $\p \cleq \q$ implies that 
\begin{equation}\label{partialsums}
 0 \leq E[ \sum_{i=1}^k (Z_i - Y_i)] = \sum_{i=1}^k q_i - \sum_{i=1}^k p_i, \quad \forall k=1,2,\ldots,M. 
\end{equation}
In particular, $\p \cleq \q$ implies that $ \d(\p) \leq \d(\q)$. 
However, condition \eqref{partialsums} is not equivalent to $\p \cleq \q$ as can be seen by the example $\p = (0.5,0.5)$ and $\q = (0.5+\e,0.5-\e)$ for any $\e\in (0,0.5)$.
\end{rem}

Our goal is to prove the following strict monotonicity results 
for excited random walks. 

\begin{thm}\label{Transmono}
 If $\vec{p} \prec \vec{q}$ and $\d(\p) > 1$, then 
\[
 P_{\p}( X_n > 0, \, \forall n>0) < P_{\q}( X_n > 0, \, \forall n>0).
\]
\end{thm}

\begin{thm}\label{Speedmono}
 If $\vec{p} \prec \vec{q}$ then either $v(\p) = v(\q) = 0$ or $v(\p) < v(\q)$.
\end{thm}

Due to the self-interacting nature of excited random walks, the natural approach to proving Theorems \ref{Transmono} and \ref{Speedmono} via a direct coupling of excited random walks fails. 
Many of the recent results for one-dimensional excited random walks use a branching processes with migration that is related to the excited random walk \cite{bsCRWspeed,bsRGCRW,kzPNERW,kmLLCRW,rrMOTERW,dkSLRERW,pLDSERW}. 
The main idea of the proofs of Theorems \ref{Transmono} and \ref{Speedmono} is that the ordering $\cleq$ on cookie environments leads to a very natural coupling of the associated branching processes with migration. 

Holmes and Salisbury introduced the partial ordering $\cleq$ in \cite{hsCombinatorial} as a way to use couplings to obtain monotonicity results on a wide variety of self-interacting random walks, including excited random walks. 
A consequence of \cite[Theorems 1.3 and 5.2]{hsCombinatorial} is that the weak monotonicity results in Theorems \ref{zTransmono} and \ref{zSpeedmono} above extend to the partial ordering $\cleq$ on cookie environments.
However, their results do not imply the strict monotonicity in Theorems \ref{Transmono} and \ref{Speedmono}.   
The couplings given in \cite{hsCombinatorial} are somewhat difficult to work with using the random walks directly,
and the main novelty in the present paper is the observation that the partial ordering $\p \cleq \q$ gives a very natural coupling of the associated branching processes with migration
that is easier to work with.

The remainder of this paper is structured as follows. 
A review of the associated (forward and backward) branching processes with migration is given in Section \ref{sec:BPWM}. Sections \ref{sec:TM} and \ref{sec:SM} then construct the couplings of the branching processes and show how they can be used to prove Theorems \ref{Transmono} and \ref{Speedmono}, respectively. 

\section{Forward and backward branching processes with migration} \label{sec:BPWM}

In this section we recall two branching processes with migration and explain their relation to the excited random walks. 
We will use the ``coin tossing'' construction of the branching processes given in \cite{kzPNERW}. For a fixed cookie environment $\p=(p_1,p_2,\ldots p_M)$ let $\{ \xi_{i,j} \}_{i \in \Z, \, j\geq 1}$ be a collection of independent Bernoulli random variables with 
\[
 \xi_{i,j} \sim
\begin{cases}
 \text{Bernoulli($p_j$)} & j \leq M \\
 \text{Bernoulli($1/2$)} & j > M. 
\end{cases}
\]
The random variables $\xi_{i,j}$ can be used to generate the path of the excited random walk. In particular, on the $j$-th visit to the site $i\in \Z$ the random walk will step right if $\xi_{i,j} = 1$ and left otherwise. 

Next we use the same Bernoulli random variables $\xi_{i,j}$ to construct the forward and backward branching processes associated with the excited random walk.

\subsection{Forward branching process}

For each fixed $i \in \Z$, we may think of the sequence $\{ \xi_{i,j} \}_{j\geq 1}$ as representing a sequence of independent trials. We say that the $j$-th trial is a \emph{success} if $\xi_{i,j} = 1$ and a \emph{failure} otherwise, and we define $S_{i,k}$ for to be the number of successes in the sequence $\{\xi_{i,j}\}_{j\geq 1}$ before the $k$-th failure. That is, 
\[
 S_{i,k} = \min \left\{ m \geq 0: \, \sum_{j=1}^{m+k} \xi_{i,j} = m \right\}. 
\]
(Note that this definition makes sense for $k=0$ as well with $S_{i,0} = 0$.) 
Then the forward branching process $\{V_i\}_{i\geq 0}$ is defined by 
\be\label{fbpdef}
V_0 = 1, \text{ and } V_{i+1} = S_{i+1,V_i} \text{ for } i\geq 0. 
\ee

The relevance of the forward branching process to the excited random walk can be seen by looking at the right excursions of the excited random walk. 
Let $T_0^+ = \inf \{ n\geq 1: \, X_n = 0 \}$ be the first return time of the excited random walk to the origin. 
For every $i\geq 0$ let $U_i$ be the number of times the random walk steps from $i$ to $i+1$ before time $T_0^+$. That is, 
\[
 U_i = \# \{ n < T_0^+ : \, X_n = i, \, X_{n+1} = i+1 \}. 
\]

\begin{lem}\label{TUVlem}
$U_i \leq V_i$ for all $i\geq 0$, with equality for all $i\geq 0$ on the event $\{X_1 = 1, \, T_0^+ < \infty \}$.
In particular, this implies that 
\be\label{TUV}
P_{\p}\left(X_1 = 1, \, T_0^+ = 2k \right) = P_{\p}\left( \sum_{i\geq 0} U_i = k \right) = p_1 P_{\p}\left( \sum_{i\geq 0} V_i = k \right), \quad \forall k\geq 1. 
\ee
\end{lem}
\begin{proof}
The first part of the lemma was proved in \cite[Section 4]{kzPNERW} and so we will only briefly explain the intuition. First note that $U_0 = 1$ if and only if $X_1 = 1$ and that $V_0 = 1$ by definition. When $X_1 = 1$ and $T_0^+<\infty$ then every jump to the right from $i\geq 1$ is matched by a jump to the left from $i+1$. If $U_i = k$ then there are $k$ jumps from $i+1$ to $i$ before $T_0^+$ and the $k$-th such jump to the left corresponds to the $k$-th failure in the sequence $\{\xi_{i+1,j}\}_{j\geq 1}$.
Therefore, it is evident that the number of jumps to the right from $i+1$ before $T_0^+$ is $U_{i+1} = S_{i+1,k}$ and so $U_i = V_i$ implies $U_{i+1} = V_{i+1}$ when $T_0^+ < \infty$.
On the other hand, if $X_1 = 1$ but $T_0^+ = \infty$ then the last jump right at every $i\geq 1$ is not matched with a corresponding jump to the left from $i+1$. Therefore, if $U_i = k$ then it may be that not all of the ``successes'' before the $k$-th ``failure'' in the sequence $\{ \xi_{i+1,j} \}_{j\geq 1}$ are used to generate jumps to the right. Thus $U_{i+1} \leq S_{i+1,k}$ and so by induction we see that $U_i \leq V_i$ on the event $\{T_0^+ = \infty \}$. 

To prove \eqref{TUV} first note that 
\be\label{UT}
 2 \sum_{i\geq 0} U_i = 
\begin{cases}
 0 & \text{on } \{X_1 = -1\} \\
 T_0^+ & \text{on } \{ X_1 = 1, \, T_0^+ < \infty\} \\
 \infty & \text{on } \{X_1 = 1, \, T_0^+ = \infty\}. 
\end{cases}
\ee
This proves the first equality in \eqref{TUV}. 
To prove the second equality in \eqref{TUV} we claim that 
\be\label{UV}
 \left\{ \sum_{i\geq 0} U_i = k \right\} = \left\{ X_1 = 1, \, \sum_{i\geq 0} V_i = k \right\}. 
\ee
Indeed, $\sum_{i\geq 0} U_i = k$ for some finite $k\geq 1$ implies that $X_1 = 1$ and $T_0^+ =2k < \infty$. Therefore, the first part of the lemma then implies that $U_i = V_i$ for all $i\geq 0$ and so $\sum_{i\geq 0} V_i = k$ as well. 
On the other hand, if $X_1 = 1$ and $\sum_{i\geq 0} V_i$ is finite then $U_0 = 1$ and it follows from the first part of the lemma that $\sum_{i\geq 0} U_i$ is finite as well. By \eqref{UT} this implies that $T_0^+$ is finite and so $\sum_{i\geq 0} U_i = \sum_{i\geq 0} V_i$ when $X_1 = 1$ and $\sum_{i\geq 0} V_i < \infty$. 
Finally, the second equality in \eqref{TUV} then follows from \eqref{UV} and the fact that the events $\{X_1 = 1\}$ and $\{\sum_{i\geq 0} V_i = k \}$ are independent since the first only depends on $\xi_{0,1}$ and the second only depends on $\{\xi_{i,j}\}_{i\geq 1, \, j\geq 1}$ (recall that $V_0 = 1$ by definition). 
\end{proof}

Let $\s_V = \inf \{ i \geq 1: V_i = 0 \}$ be the lifetime of the forward branching process $V_i$. Since $\sum_{i\geq 0} V_i < \infty$ if and only if $\s_V < \infty$, applying Lemma \ref{TUVlem} implies the following corollary. 
\begin{cor}\label{Transform}
 $P_{\p}( X_n > 0, \, \forall n > 0 ) = p_1 P_{\p}( \s_V = \infty )$. 
\end{cor}
\begin{proof}
 Obviously, $P_{\p}( X_n > 0, \, \forall n > 0 ) = P_{\p}( X_1 = 1, \, T_0^+ = \infty )$. Then, Lemma \ref{TUVlem} implies that 
\begin{align*}
 P_{\p}( X_n > 0, \, \forall n > 0 ) &= P_{\p}(X_1 = 1) - P_{\p}( X_1 = 1, \, T_0^+ < \infty ) \\
&= p_1 - p_1 P_{\p}\left( \sum_{i\geq 0} V_i < \infty \right) = p_1 P_{\p}\left( \sum_{i\geq 0} V_i = \infty \right). 
\end{align*}
The proof is concluded by noting that $\{ \sum_{i\geq 0} V_i = \infty \} = \{ \s_V = \infty \}$ since $0$ is an absorbing state for the Markov chain $V_i$. 
\end{proof}

\subsection{Backward Branching Process}

The backward branching process that we will construct differs from the above constructed forward branching process in two ways. First, we will reverse the role of ``failures'' and ``successes.'' Secondly, there will be an extra immigrant before reproduction in each generation. To make this precise, let $F_{i,k}$ be the number of failures in the sequence $\{\xi_{i,j} \}_{j\geq 1}$ before the $k$-th success. That is, 
\[
 F_{i,k} = \min\left\{ m \geq 0: \, \sum_{j=1}^{m+k} \xi_{i,j} = k \right\}.
\]
(Note that the above definition gives $F_{i,0} = 0$.)
The backward branching process $\{Z_i\}_{i\geq 0}$ is defined by 
\[
 Z_0 = 0, \quad \text{and } Z_{i+1} = F_{i+1, \, Z_i + 1} \text{ for every } i \geq 0. 
\]

The relevance of the backward branching process to the excited random walk can be seen by examining the hitting times $T_n = \inf \{ k\geq 0: \, X_k = n \}$. 
For $n\geq 1$ and $x \leq n$ let $D_x^n$ be the number of times the random walk jumps left from $x$ to $x-1$ before time $T_n$. That is, 
\[
 D_x^n = \# \left\{ k < T_n : \, X_k = x, \, X_{k+1} = x-1 \right\}. 
\]
\begin{lem}\label{ZDlem}
 Suppose that the excited random walk is either recurrent or transient to the right (that is $\d(\p) \geq - 1$). Then for any $n\geq 1$, the sequence $(D_n^n,D_{n-1}^n,\ldots, D_1^n, D_0^n)$ has the same distribution as $(Z_0,Z_1,\ldots, Z_{n-1}, Z_n)$. 
\end{lem}
\begin{rem}
 This relation between the backward branching process and the excited random walk was first shown in \cite{bsCRWspeed} for deterministic cookie environments, but the general idea goes back at least to \cite{kksStable} where it was used to analyze random walks in a random environment. It should also be noted that Lemma \ref{ZDlem} also holds for excited random walks when the cookie environment is spatially i.i.d.\ \cite{kzPNERW}.
\end{rem}

We will omit the proof of Lemma \ref{ZDlem} since it is well known (see the above references) and the ideas are similar to the proof of Lemma \ref{TUVlem}. We remark that (in contrast to Lemma \ref{TUVlem}) it is not true that the two sequences are identical, but only that they have the same distribution and that this follows from the fact that the sequence of Bernoulli trials at each site $\{\xi_{i,j}\}_{i\in\Z, \, j\geq 1}$ is i.i.d.\ in $i$. 

Lemma \ref{ZDlem} can be used to relate (the distribution of) the backward branching process to the hitting times $T_n$. However, the correspondence is slightly simpler if one instead considers the amount of time the random walk spends to the right of the origin prior to hitting $n$. That is, 
\[
 \widetilde{T}_n = \# \{ k < T_n : \, X_k \geq 0 \}. 
\]
 
\begin{cor}
 Suppose that the excited random walk is either recurrent or transient to the right (that is $\d(\p) \geq -1$). Then,
\be\label{TZ}
 \widetilde{T}_n 
\overset{\text{law}}{=} n + 2 \sum_{i=0}^{n-1} Z_i + Z_n
\ee
\end{cor}
\begin{proof}
 First note that on the event $\{ T_n < \infty \}$,
\be\label{TDrep}
 T_n = n + 2 \sum_{x\leq n} D_x^n
\ee
Indeed the $n$ counts the first step from $x$ to $x+1$ for each $x=0,1,\ldots, n-1$, and every step from $x$ to $x-1$ before $T_n$ is matched with a later corresponding step from $x-1$ to $x$. 
Similarly, by only counting the steps that begin at $x \geq 0$ we obtain that
\[
  \widetilde{T}_n = n + 2 \sum_{x=1}^n D_x^n + D_0^n.
\]
The proof is then completed by applying Lemma \ref{ZDlem} and noting that the assumption that $\d(\p) \geq -1$ implies that $T_n < \infty$ with probability one.
\end{proof}

\section{Proof of Theorem \ref{Transmono}} \label{sec:TM}

In this section we will use the forward branching process $V_i$ defined above to prove Theorem \ref{Transmono}. 

\begin{proof}[Proof of Theorem \ref{Transmono}]
 As noted in Remark \ref{ineqrem}, $\p \prec \q$ implies that $p_1 \leq q_1$. Thus, by Corollary \ref{Transform} it is enough to prove that 
\be\label{Pineq}
 \d(\p) > 1 \text{ and } \p \prec \q \quad \Longrightarrow \quad P_{\p}( \s_V = \infty ) < P_{\q}( \s_V = \infty). 
\ee
To this end, we will use the relation $\p\prec \q$ to give a coupling of the forward branching processes in the cookie environments $\p$ and $\q$, respectively. 
In particular, let $$\{(\mathbf{Y}_i, \, \mathbf{Z}_i )\}_{i \in \Z} = \{((Y_{i,1},\ldots,Y_{i,M}), (Z_{i,1},\ldots,Z_{i,M})) \}_{i\in \Z}$$ be an i.i.d.\ sequence of joint random variables with each $(\mathbf{Y}_i, \, \mathbf{Z}_i)$ having the same joint distribution as  
$(\mathbf{Y},\, \mathbf{Z})$ given in the definition of the partial ordering $\p \prec \q$. Also, let $\{ \mathbf{B}_i \}_{i\in\Z} = \{ B_{i,j} \}_{i \in \Z, \, j \geq 1}$ be an i.i.d.\ collection of Bernoulli($1/2$) random variables that is independent of $\{(\mathbf{Y}_i, \mathbf{Z}_i)\}_{i\in \Z}$. Finally, we will construct $\Xi = \{ \Xi_i \} = \{ \xi_{i,j} \}_{i\in \Z, j\geq 1}$ and $ \Xi' = \{ \Xi_i' \} = \{ \xi_{i,j}' \}_{i\in \Z, j\geq 1}$ by letting $\Xi_i = (\mathbf{Y}_i, \mathbf{B}_i)$ and $\Xi_i' = (\mathbf{Z}_i, \mathbf{B}_i)$. That is,
\[
 \xi_{i,j} = 
\begin{cases}
 Y_{i,j} & \text{if } j\leq M \\
 B_{i,j-M} & \text{if } j> M
\end{cases}
\quad\text{and}\quad
 \xi_{i,j}' = 
\begin{cases}
 Z_{i,j} & \text{if } j\leq M \\
 B_{i,j-M} & \text{if } j> M.
\end{cases}
\]
We will denote the joint distribution of $(\Xi,\Xi')$  by $P_{\p,\q}$. 
Note that this coupling has the following properties. 
\begin{itemize}
 \item The marginal distributions of $\Xi$ and $\Xi'$ are $P_{\p}$ and $P_{\q}$, respectively. 
 \item $\{(\Xi_i, \Xi_i')\}_{i\in \Z}$ is an i.i.d.\ sequence under $P_{\p,\q}$. 
 \item For any $i\in \Z$ and $m\geq 1$, $\sum_{j=1}^m \xi_{i,j} \leq \sum_{j=1}^m \xi_{i,j}'$, $P_{\p,\q}$\,-a.s. 
 \item There exists an $m \leq M$ such that 
\be\label{xisum}
P_{\p,\q}\left( \sum_{j=1}^m \xi_{i,j} < \sum_{j=1}^m \xi_{i,j}' \right) > 0. 
\ee
\end{itemize}

With the above construction of $\Xi$ and $\Xi'$ we can now give a coupling of the forward branching processes in cookie environments $\p$ and $\q$, respectively. Let, $S_{i,k}$ and $S_{i,k}'$ be the number of successes before the $k$-th failure in the sequences of Bernoulli trials $\{\xi_{i,j}\}_{j\geq 1}$ and $\{\xi_{i,j}'\}_{j\geq 1}$, respectively, and let
\[
 V_0 = V_0' = 1, \quad\text{and}\quad V_{i+1} = S_{i+1,V_i}, \quad  V_{i+1}' = S_{i+1,V_i'}'.
\]
The above properties of the coupling $(\Xi,\Xi')$ imply that
$S_{i,k} \leq S_{i,k}'$ for all $i\in \Z$ and $k\geq 1$. Since 
\[
V_i \leq V_i' \quad \Longrightarrow \quad 
 V_{i+1} = S_{i+1,V_i} \leq S_{i+1,V_i'} \leq S_{i+1,V_i'}' = V_{i+1}', 
\]
it follows by induction that $V_i \leq V_i'$ for all $i\geq 0$. 

Let $\s_V = \inf \{ i\geq 1: V_i = 0 \}$ and $\s_V' = \inf \{ i\geq 1: V_i' = 0 \}$ be the lifetimes of the processes $V_i$ and $V_i'$, respectively. Then, since $V_i \leq V_i'$ it follows that $ \s_V \leq \s_V'$, and so to prove \eqref{Pineq} it is enough to show
\be\label{PsV}
 P_{\p,\q}(\s_V < \infty, \, \s_V' = \infty) > 0. 
\ee
To this end, note that for any $k\geq 1$, 
\begin{align}
 P_{\p,\q}(\s_V < \infty, \, \s_V' = \infty) 
&\geq P_{\p,\q}( k = V_1 \leq V_1', \, V_2 = 0 < V_2', \, \s_V' = \infty ) \nonumber \\
&\geq P_{\p,\q}( k = V_1 \leq V_1', \, V_2 = 0 < V_2') P_{\q}(\s_V = \infty)  \nonumber \\
&\geq P_{\p,\q}( V_1 = k \leq V_1', \, S_{2,k} = 0 < S_{2,k}') P_{\q}(\s_V = \infty)  \nonumber \\
&= P_{\p}( V_1 = k )P_{\p,\q}(S_{2,k} = 0 < S_{2,k}') P_{\q}(\s_V = \infty), \label{3p}
\end{align}
where in the second inequality we use the Markov property and the fact that $P_{\q}(\s_V = \infty |\, V_0 = m )$ is non-decreasing in $m$ (which can be easily seen from the construction of the process $V_i$).
It was shown in \cite{kzPNERW} that $P_{\q}(\s_V = \infty) > 0$ if and only if $\d(\q) > 1$ (in \cite{kzPNERW} this fact was then used to prove part \ref{transright} of Theorem \ref{rtspeed}). Since we assumed that $\d(\q) \geq \d(\p) > 1$, using \eqref{3p} we will be able to conclude that \eqref{PsV} holds if we can show that 
\be\label{Sineq}
 P_{\p,\q}(S_{2,k} = 0 < S_{2,k}') > 0 \quad \text{ for some } k\geq 1. 
\ee

To show \eqref{Sineq}, let 
\[
 k = \min \left\{ m\geq 1: P_{\p,\q}\left( \sum_{j=1}^m \xi_{i,j} < \sum_{j=1}^m \xi_{i,j}'  \right) > 0 \right\}.
\]
Then it must be the case that $P_{\p,\q}( \xi_{i,j} = \xi_{i,j}', \, j=1,2,\ldots k-1 ) = 1$. In particular, this implies that $0 = E[\sum_{j=1}^m (\xi_{i,j} - \xi_{i,j}') ] = \sum_{j=1}^m (p_i - q_i)$ for all $m < k$ while 
$0 > E[\sum_{j=1}^k (\xi_{i,j} - \xi_{i,j}') ] = \sum_{j=1}^k (p_i - q_i)$. Thus, this choice of $k$ implies that $p_j = q_j$ for all $j<k$ and $p_k < q_k$. 
Therefore, we have that 
\begin{align}
 P_{\p,\q}(S_{2,k} = 0 < S_{2,k}') &= P_{\p,\q}\left( \sum_{j=1}^{k} \xi_{i,j} = 0 < \sum_{j=1}^{k} \xi_{i,j}' \right)   \nonumber     \\
&= P_{\p,\q}\left( \sum_{j=1}^{k-1} \xi_{i,j} = 0, \, \xi_{i,k} = 0, \, \xi_{i,k}'=1 \right) \nonumber \\
&= P_{\p,\q}\left( \sum_{j=1}^{k-1} \xi_{i,j}' = 0, \xi_{i,k}'=1 \right) - P_{\p,\q}\left( \sum_{j=1}^{k-1} \xi_{i,j} = 0, \xi_{i,k}=1 \right) \nonumber \\
&= (1-q_1)\cdots(1-q_{k-1})q_k - (1-p_1)\cdots(1-p_{k-1})p_k. \label{diff}
\end{align}
where in the second and third equalities above we are using that $\xi_{i,j} = \xi_{i,j}'$ for all $j < k$. Then since our choice of $k$ implies that $p_j = q_j$ for $j<k$ and $p_k < q_k$, \eqref{diff} is strictly positive. 
This completes the proof of \eqref{Sineq} and thus also the proof of Theorem \ref{Transmono}. 
\end{proof}

\section{Proof of Theorem \ref{Speedmono}} \label{sec:SM}

If $\d(\p) > 1$, then the excited random walk is transient to the right. Therefore, the total time spent to the left of the origin is finite. In particular, the total number of jumps to the left from the origin $D_0^\infty = \lim_{n\ra\infty} D_0^n $ is finite $P_{\p}$\,-a.s. Then, Lemma \ref{ZDlem} implies that $Z_n$ converges in distribution as $n\ra\infty$ and we will let $P_{\p}^\infty$ denote this limiting distribution. 
As was noted in \cite{bsCRWspeed}, since the backward branching process is an irreducible Markov chain (as can be easily seen from the definition) which converges in distribution, it follows that $Z_i$ is positive recurrent and that $P_{\p}^\infty$ is the unique invariant measure for the Markov chain $\{Z_i\}_{i\geq 0}$.  
Thus, an ergodic theorem for positive recurrent, irreducible Markov chains implies that  
\be\label{Zavg}
 \lim_{n\ra\infty} \frac{1}{n} \sum_{i=0}^{n-1} Z_i = E_{\p}^\infty[ Z_0 ], \qquad P_{\p}\text{-a.s.} 
\ee
Using this and the connection of the backwards branching process with the hitting times of the excited random walk, it is easy to obtain the following formula for the limiting speed of the excited random walk.

\begin{lem}[Corollary 2.3 in \cite{bsCRWspeed}]\label{vform}
 Suppose that $\d(\p) > 1$. Then, 
\[
 v(\p) := \lim_{n\ra\infty} X_n/n = \frac{1}{1 + 2 E_{\p}^\infty[ Z_0 ]}, \quad P_{\p}\text{-a.s.}
\]
\end{lem}
\begin{rem}
 We have not assumed that $E_{\p}^\infty[ Z_0 ] < \infty$, and in fact this is only true if $\d(\p) > 2$. However, the formula above still holds even if $\d(\p) \in (1,2]$ since the limiting speed is $v(\p) = 0$ in that case \cite{bsCRWspeed,kzPNERW}. 
\end{rem}

We are now ready to give the proof of Theorem \ref{Speedmono}

\begin{proof}[Proof of Theorem \ref{Speedmono}]
Recall from Remark \ref{ineqrem} that $\p \prec \q$ implies that $\d(\p) \leq \d(\q)$.  We will assume that either $\d(\p)>2$ or $\d(\q) < -2$ since otherwise the conclusion of the Theorem \ref{Speedmono} is obvious from Theorem \ref{rtspeed}. Without loss of generality we will assume that $\d(\q) \geq \d(\p) > 2$ since the case $ \d(\p) \leq \d(\q) < -2$ can be handled by a symmetry argument. 
By the formula for $v(\p)$ in Lemma \ref{vform} it is enough to prove
\be\label{EZineq}
 \d(\p) > 2 \text{ and } \p \prec \q \quad \Longrightarrow \quad E_{\p}^\infty[Z_0] > E_{\q}^\infty[Z_0]. 
\ee
Note that since $\d(\q)\geq \d(\p) > 2$, the limiting speeds $v(\p),v(\q) > 0$ and thus the expectations $E_{\p}^\infty[Z_0]$ and $E_{\q}^\infty[Z_0]$ are both finite.

We will use the assumption $\p \prec \q$ to give a coupling of the backward branching processes in cookie environments $\p$ and $\q$, respectively. 
Let $(\Xi, \Xi')$ be the coupled families of Bernoulli random variables defined as in the proof of Theorem \ref{Transmono}. 
 Let, $F_{i,k}$ and $F_{i,k}'$ be the number of failures before the $k$-th success in the sequences of Bernoulli trials $\{\xi_{i,j}\}_{j\geq 1}$ and $\{\xi_{i,j}'\}_{j\geq 1}$, respectively, and let
\[
 Z_0 = Z_0' = 0, \quad\text{and}\quad Z_{i+1} = F_{i+1,Z_i + 1}, \quad  Z_{i+1}' = F_{i+1,Z_i' + 1}'.
\]
The properties of the coupling $(\Xi,\Xi')$ imply that
$F_{i,k} \geq F_{i,k}'$ for all $i\in \Z$ and $k\geq 1$ and thus also that $Z_i \geq Z_i'$ for all $i\geq 0$. 


Now, since $Z_i \geq Z_i' \geq 0$ and $Z_i$ is a positive recurrent Markov chain, it follows that $\{(Z_i,Z_i')\}$ is a positive recurrent Markov chain. Moreover, it is easy to see that this Markov chain is aperiodic and irreducible and therefore $(Z_n,Z_n')$ converges in distribution as $n\ra\infty$. We will denote this limiting distribution by $P_{\p,\q}^\infty$ (note that the marginal distributions of $P_{\p,\q}^\infty$ are necessarily $P_{\p}^\infty$ and $P_{\q}^\infty$). Thus, 
\be\label{avglb}
 \lim_{n\ra\infty} \frac{1}{n} \sum_{i=0}^{n-1} (Z_i - Z_i') \geq \lim_{n\ra\infty} \frac{1}{n} \sum_{i=0}^{n-1} \ind{Z_i > Z_i'} = P_{\p,\q}^\infty (Z_0 > Z_0'). 
\ee
Since $P_{\p,\q}^\infty$ is a stationary distribution we can calculate a lower bound on this probability by 
\be\label{Pinflb}
 P_{\p,\q}^\infty (Z_0 > Z_0') \geq P_{\p,\q}^\infty ( Z_0 = k-1, \, Z_1 > Z_1') \geq P_{\p}^\infty( Z_0 = k-1) P_{\p,\q}( F_{1,k} > F_{1,k}' ). 
\ee
In the last inequality we used that if $Z_0' \leq Z_0 = k-1$ then $Z_1 = F_{1,k}$ and $Z_1' \leq F_{1,k}'$.  
It follows from \eqref{xisum} in the coupling of $\Xi$ and $\Xi'$ that there exists a $k \leq M$ such that the last probability in \eqref{Pinflb} is non-zero.
Combining \eqref{avglb} and \eqref{Pinflb} we obtain
\[
  0 < P_{\p,\q}^\infty (Z_0 > Z_0') \leq \lim_{n\ra\infty} \frac{1}{n} \sum_{i=0}^{n-1} (Z_i - Z_i') = E_{\p}^\infty[Z_0] - E_{\q}^\infty[Z_0], 
\]
where the last equality follows from \eqref{Zavg} 
and the fact that $E_{\p}^\infty[Z_0], \, E_{\q}^\infty[Z_0] < \infty$ since we are only considering the case $\d(\p),\d(\q) > 2$. 
Thus \eqref{EZineq} holds, which completes the proof of Theorem \ref{Speedmono}. 
\end{proof}

\bibliographystyle{alpha}
\bibliography{CookieRW}

\end{document}